\documentclass{amsart}

\usepackage{amsfonts,mathrsfs}
\usepackage{amsmath,amssymb,amsthm, amscd, bm}
\usepackage{tikz}
\usetikzlibrary{arrows,calc,matrix,topaths,positioning,scopes,shapes,decorations,decorations.markings} 
\usepackage{dsfont}
\usepackage{epsfig}
\usepackage{tikz-cd}
\usepackage{hyperref}
\usepackage{fullpage}
\usepackage[foot]{amsaddr}
\usepackage{adjustbox}
\usepackage{enumitem}

\author{Julien Korinman}
\address{Institut Montpelli\'erain Alexander Grothendieck - UMR 5149 Universit\'e de Montpellier. Place Eug\'ene Bataillon, 34090 Montpellier France}
\email{julien.korinman@gmail.com}
\urladdr{https://sites.google.com/site/homepagejulienkorinman/}

\subjclass{$57$R$56$,  $57$M$25$.}


\def\restriction#1#2{\mathchoice
              {\setbox1\hbox{${\displaystyle #1}_{\scriptstyle #2}$}
              \restrictionaux{#1}{#2}}
              {\setbox1\hbox{${\textstyle #1}_{\scriptstyle #2}$}
              \restrictionaux{#1}{#2}}
              {\setbox1\hbox{${\scriptstyle #1}_{\scriptscriptstyle #2}$}
              \restrictionaux{#1}{#2}}
              {\setbox1\hbox{${\scriptscriptstyle #1}_{\scriptscriptstyle #2}$}
              \restrictionaux{#1}{#2}}}
\def\restrictionaux#1#2{{#1\,\smash{\vrule height .8\ht1 depth .85\dp1}}_{\,#2}}

\newcommand{\quotient}[2]{{\raisebox{.2em}{$#1$}\left/\raisebox{-.2em}{$#2$}\right.}}

\newcommand{\sslash}{\mathbin{/\mkern-6mu/}}
\newcommand{\Hom}{\operatorname{Hom}}
\newcommand{\tr}{\operatorname{tr}}

\newcommand{\SL}{\operatorname{SL}}
\newcommand{\SU}{\operatorname{SU}}

\newcommand{\Span}{\operatorname{Span}}

\newcommand{\Specm}{\operatorname{MaxSpec}}
\newcommand{\Spec}{\operatorname{Spec}}
\newcommand{\Mod}{\operatorname{Mod}}

\newcommand{\Mat}{\operatorname{Mat}}

\newcommand{\Alg}{\operatorname{Alg}}

\newcommand{\heightexch}[3]{
	\begin{tikzpicture}[baseline=-0.4ex,scale=0.5, >=stealth]
	\draw [fill=gray!60,gray!45] (-.7,-.75)  rectangle (.4,.75)   ;
	\draw[#1] (0.4,-0.75) to (.4,.75);
	\draw[line width=1.2] (0.4,-0.3) to (-.7,-.3);
	\draw[line width=1.2] (0.4,0.3) to (-.7,.3);
	\draw (0.65,0.3) node {\scriptsize{$#2$}}; 
	\draw (0.65,-0.3) node {\scriptsize{$#3$}}; 
	\end{tikzpicture}
}
\newcommand{\heightcurve}{
\begin{tikzpicture}[baseline=-0.4ex,scale=0.5]
\draw [fill=gray!20,gray!45] (-.7,-.75)  rectangle (.4,.75)   ;
\draw[-] (0.4,-0.75) to (.4,.75);
\draw[line width=1.2] (-.7,-0.3) to (-.4,-.3);
\draw[line width=1.2] (-.7,0.3) to (-.4,.3);
\draw[line width=1.15] (-.4,0) ++(-90:.3) arc (-90:90:.3);
\end{tikzpicture}
}

\begin{document}

\theoremstyle{plain}
\newtheorem{theorem}{Theorem}[section]
\newtheorem{proposition}[theorem]{Proposition}
\newtheorem{corollary}[theorem]{Corollary}
\newtheorem{lemma}[theorem]{Lemma}
\theoremstyle{definition}
\newtheorem{notations}[theorem]{Notations}
\newtheorem{convention}[theorem]{Convention}
\newtheorem{problem}[theorem]{Problem}
\newtheorem{definition}[theorem]{Definition}
\theoremstyle{remark}
\newtheorem{remark}[theorem]{Remark}
\newtheorem{conjecture}[theorem]{Conjecture}
\newtheorem{example}[theorem]{Example}
\newtheorem{strategy}[theorem]{Strategy}
\newtheorem{question}[theorem]{Question}

\title[Skein modules of closed $3$-manifolds define line bundles over character varieties]{Skein modules of closed $3$-manifolds define line bundles over character varieties}

\date{}
\maketitle


\begin{abstract} 

Let $M$ be a closed $3$-manifold, $\zeta$ a root of unity of odd order and $\mathcal{S}_{\zeta}(M)$ the Kauffman-bracket skein module. Using the Frobenius morphism, we can see $\mathcal{S}_{\zeta}(M)$ as the space of global sections of a coherent sheaf over the $\SL_2$ character scheme of $M$. We prove that when the character scheme is reduced, this sheaf is a line bundle.
\end{abstract}


\section{Introduction}

Let $M$ be a closed connected $3$-manifold, $\zeta$ a root of unity of odd order $N\geq 3$ and $\mathcal{S}_{\zeta}(M)$ the Kauffman-bracket skein module. The \textit{Frobenius morphism} is a linear map $Fr: \mathcal{S}_{+1}(M)\to \mathcal{S}_{\zeta}(M)$, defined in \cite{BonahonWong1, LeKauffmanBracket}, whose image consists in "transparent" elements. This transparency property permits to endow $\mathcal{S}_{\zeta}(M)$ with a structure of $\mathcal{S}_{+1}(M)$-module. Let $X_{\SL_2}(M)$ be the $\SL_2$ character scheme of $M$. Fix a spin structure $S$ on $M$. Using \cite{Bullock,PS00, Barett}, we have an isomorphism  $\Psi^S: \mathcal{O}[X_{\SL_2}(M)] \xrightarrow{\cong} \mathcal{S}_{+1}(M)$ so $\mathcal{S}_{\zeta}(M)$ becomes a finitely generated module over $ \mathcal{O}[X_{\SL_2}(M)]$, thus defines a coherent sheaf $\mathscr{L}^M \to X_{\SL_2}(M)$. The main result of this paper is the following:

\begin{theorem}\label{main_theorem_intro} Suppose that $X_{\SL_2}(M)$  is reduced. Then the sheaf $\mathscr{L}^M \to X_{\SL_2}(M)$ is a line bundle.
\end{theorem}

This sheaf was first considered in \cite{FKL_GeometricSkein} where its is proved, under the  assumption that $X_{\SL_2}(M)$ is reduced, that the restriction of $\mathscr{L}^M$ to the locus of classes $[\rho]$ of irreducible representations $\rho: \pi_1(M)\to \SL_2$ is a line bundle. 
This result was strengthen in \cite{FFK_SkeinIrrep} where the same result is proved without the reduced assumption.
It was then proved in \cite{KojuKaruo_Azumaya} that the restriction of $\mathscr{L}^M$ to the locus of classes of diagonal non central representations is a line bundle as well. So the only remaining case is to prove that $\mathscr{L}^M$ is a line bundle in the neighborhood of central representations,  i.e. for classes of representations $\rho: \pi_1(M) \to \SL_2$ which satisfy $\rho(\gamma)=\pm \mathds{1}$ for all $\gamma \in \pi_1(M)$.  Based on the original work in \cite{DetcherryKalfagianniSikora_SkeinDim}, it was proved in \cite{KojuKaruo_Azumaya} that Theorem \ref{main_theorem_intro} holds when $X_{\SL_2}(M)$ is finite, reduced and when the skein module $\mathcal{S}_A(M)$ for $A$ a generic parameter in $\mathbb{Q}[A^{\pm 1}]$ satisfies an additional "tame" property which is very hard to check in practice. 
\vspace{2mm}
\par We will consider another line bundle based this time on stated skein modules. Let $M^*$ be the marked $3$-manifold $M^*=(M\setminus \mathring{B}^3, \{a\})$ obtained by removing from $M$ an open ball and placing one marking on the boundary of this ball. Let $\mathcal{S}_{\zeta}(M^*)$ be the stated skein module of $M^*$. By \cite{KojuQuesneyClassicalShadows, BloomquistLe} we still have a Chebyshev-Frobenius morphism $Fr_*: \mathcal{S}_{+1}(M^*) \to \mathcal{S}_{\zeta}(M^*)$ which endows $\mathcal{S}_{\zeta}(M^*)$ with a structure of $\mathcal{S}_{+1}(M^*)$ module. By \cite{KojuMurakami_QCharVar}, the spin structure $S$ defines an isomorphism $\Psi^{S}_*: \mathcal{O}[R_{\SL_2}(M)] \xrightarrow{\cong} \mathcal{S}_{+1}(M^*)$ where $R_{\SL_2}(M):= \Hom(\pi_1(M), \SL_2)$ is the representation scheme of $M$. Therefore $\mathcal{S}_{\zeta}(M^*)$ defines a coherent sheaf $\mathscr{I}^{M} \to R_{\SL_2}(M)$. 
Let $\pi: R_{\SL_2}(M) \to X_{\SL_2}(M)$ be the projection map and $\pi' : \pi^*\mathscr{L}^M \to \mathscr{I}^{M}$ the bundle morphism defined by the inclusion morphism $\mathcal{S}_{\zeta}(M)\to \mathcal{S}_{\zeta}(M^*)$. 

\begin{theorem}\label{theorem2_intro}
\begin{enumerate}
\item When $A\in k^{\times}$ is an invertible element of an arbitrary commutative unital ring $k$,  the inclusion morphism $\mathcal{S}_A(M) \to \mathcal{S}_A(M^*)$ is injective.
\item Suppose that the representation scheme $R_{\SL_2}(M)$ is reduced.
 Then the sheaf $\mathscr{I}^{M} \to R_{\SL_2}(M)$ is a line bundle
 .
 \end{enumerate}
 \end{theorem}

 The fact that $\iota: \mathcal{S}_A(M) \to \mathcal{S}_A(M^*)$ is injective contrasts with the fact, proved in \cite{CostantinoLe_SSkeinTQFT}, that when $\mathcal{M}= (M, \mathcal{N})$ is a marked $3$-manifold having a connected component containing two boundary edges separated by a separating embedded sphere, then the class of the empty link in $\mathcal{S}_A(\mathcal{M})$ has torsion so the inclusion $\iota: \mathcal{S}_A(M) \to \mathcal{S}_A(\mathcal{N})$ is not injective in this case.

\textit{Acknowledgments.} The author thanks D.Calaque, F.Costantino,  C.Frohman,  B.Haioun and T.Q.T. L\^e for valuable conversations. He acknowledges support from  the European Research Council (ERC DerSympApp) under the European Union’s Horizon 2020 research and innovation program (Grant Agreement No. 768679).

\section{Stated skein modules and coherent sheaves}

\subsection{Stated skein modules}

Let us briefly review the definition and main properties of stated skein modules. A \textbf{marked }$3$\textbf{ manifold} is a pair $(M, \mathcal{N})$ where $M$ is a compact oriented $3$-manifold and $\mathcal{N}\subset \partial M$ is a finite set of pairwise disjoint embedded oriented arcs, named \textit{boundary edges}, in the boundary of $M$. An \textit{embedding of marked }$3$ \textit{ manifolds} $f: (M, \mathcal{N}) \to (M', \mathcal{N}')$ is an oriented embedding $f: M\to M'$ which induces an oriented embedding $f: \mathcal{N} \to \mathcal{N}'$. Marked $3$-manifolds with embeddings form a category $\mathcal{M}$. 
\par A \textit{tangle} $T$ in $(M, \mathcal{N})$ is a framed $1$-submanifold $T\subset M$ with $\partial T \subset \mathcal{N}$ such that if $v\in \partial T$ lies in a boundary edge $a\in \mathcal{N}$, then the framing of $T$ at $v$ points in the direction of $a$. A \textit{state} on $T$ is a map $s: \partial T \to \{\pm \}$ and a \textit{stated tangle} is a pair $(T,s)$. 

\begin{definition}(Stated skein modules) Let $k$ be a commutative unital ring and $A^{1/2}\in k^{*}$ an invertible element, whose square will be denoted by $A$. 
 The \textbf{stated skein module} $\mathcal{S}_A(M, \mathcal{N})$ is the quotient of the $k$-module freely generated by isotopy classes of stated tangles in $(M, \mathcal{N})$ by the following skein relations: 
$$
\begin{tikzpicture}[baseline=-0.4ex,scale=0.5,>=stealth]	
\draw [fill=gray!45,gray!45] (-.6,-.6)  rectangle (.6,.6)   ;
\draw[line width=1.2,-] (-0.4,-0.52) -- (.4,.53);
\draw[line width=1.2,-] (0.4,-0.52) -- (0.1,-0.12);
\draw[line width=1.2,-] (-0.1,0.12) -- (-.4,.53);
\end{tikzpicture}
=A
\begin{tikzpicture}[baseline=-0.4ex,scale=0.5,>=stealth] 
\draw [fill=gray!45,gray!45] (-.6,-.6)  rectangle (.6,.6)   ;
\draw[line width=1.2] (-0.4,-0.52) ..controls +(.3,.5).. (-.4,.53);
\draw[line width=1.2] (0.4,-0.52) ..controls +(-.3,.5).. (.4,.53);
\end{tikzpicture}
+A^{-1}
\begin{tikzpicture}[baseline=-0.4ex,scale=0.5,rotate=90]	
\draw [fill=gray!45,gray!45] (-.6,-.6)  rectangle (.6,.6)   ;
\draw[line width=1.2] (-0.4,-0.52) ..controls +(.3,.5).. (-.4,.53);
\draw[line width=1.2] (0.4,-0.52) ..controls +(-.3,.5).. (.4,.53);
\end{tikzpicture}
\hspace{.5cm}
\text{ and }\hspace{.5cm}
\begin{tikzpicture}[baseline=-0.4ex,scale=0.5,rotate=90] 
\draw [fill=gray!45,gray!45] (-.6,-.6)  rectangle (.6,.6)   ;
\draw[line width=1.2,black] (0,0)  circle (.4)   ;
\end{tikzpicture}
= -(A^2+A^{-2}) 
\begin{tikzpicture}[baseline=-0.4ex,scale=0.5,rotate=90] 
\draw [fill=gray!45,gray!45] (-.6,-.6)  rectangle (.6,.6)   ;
\end{tikzpicture}
$$

\begin{equation*} 
\begin{tikzpicture}[baseline=-0.4ex,scale=0.5,>=stealth]
\draw [fill=gray!45,gray!45] (-.7,-.75)  rectangle (.4,.75)   ;
\draw[->] (0.4,-0.75) to (.4,.75);
\draw[line width=1.2] (0.4,-0.3) to (0,-.3);
\draw[line width=1.2] (0.4,0.3) to (0,.3);
\draw[line width=1.1] (0,0) ++(90:.3) arc (90:270:.3);
\draw (0.65,0.3) node {\scriptsize{$+$}}; 
\draw (0.65,-0.3) node {\scriptsize{$+$}}; 
\end{tikzpicture}
=
\begin{tikzpicture}[baseline=-0.4ex,scale=0.5,>=stealth]
\draw [fill=gray!45,gray!45] (-.7,-.75)  rectangle (.4,.75)   ;
\draw[->] (0.4,-0.75) to (.4,.75);
\draw[line width=1.2] (0.4,-0.3) to (0,-.3);
\draw[line width=1.2] (0.4,0.3) to (0,.3);
\draw[line width=1.1] (0,0) ++(90:.3) arc (90:270:.3);
\draw (0.65,0.3) node {\scriptsize{$-$}}; 
\draw (0.65,-0.3) node {\scriptsize{$-$}}; 
\end{tikzpicture}
=0,
\hspace{.2cm}
\begin{tikzpicture}[baseline=-0.4ex,scale=0.5,>=stealth]
\draw [fill=gray!45,gray!45] (-.7,-.75)  rectangle (.4,.75)   ;
\draw[->] (0.4,-0.75) to (.4,.75);
\draw[line width=1.2] (0.4,-0.3) to (0,-.3);
\draw[line width=1.2] (0.4,0.3) to (0,.3);
\draw[line width=1.1] (0,0) ++(90:.3) arc (90:270:.3);
\draw (0.65,0.3) node {\scriptsize{$+$}}; 
\draw (0.65,-0.3) node {\scriptsize{$-$}}; 
\end{tikzpicture}
=A^{-1/2}
\begin{tikzpicture}[baseline=-0.4ex,scale=0.5,>=stealth]
\draw [fill=gray!45,gray!45] (-.7,-.75)  rectangle (.4,.75)   ;
\draw[-] (0.4,-0.75) to (.4,.75);
\end{tikzpicture}
\hspace{.1cm} \text{ and }
\hspace{.1cm}
A^{1/2}
\heightexch{->}{-}{+}
- A^{5/2}
\heightexch{->}{+}{-}
=
\heightcurve.
\end{equation*}
 \end{definition}
An embedding  $f: (M, \mathcal{N}) \to (M', \mathcal{N}')$ induces a linear map $f_*: \mathcal{S}_A(M, \mathcal{N}) \to \mathcal{S}_A(M', \mathcal{N}')$ defined by $f_* ([T,s]):= [f(T), s\circ f^{-1}]$ and we get a symmetric monoidal functor $\mathcal{S}_A: (\mathcal{M}, \sqcup) \to (\Mod_k, \otimes)$ since $\mathcal{S}_A(M_1 \sqcup M_2)\cong \mathcal{S}_A(M_1)\otimes_k \mathcal{S}_A(M_2)$. 
\par A \textbf{marked surface} is a pair $\mathbf{\Sigma}=(\Sigma, \mathcal{A})$ where $\Sigma$ is a compact oriented surface and $\mathcal{A}\subset \partial \Sigma$ a finite set of points in the boundary of $\Sigma$. It defines a marked $3$-manifold $\mathbf{\Sigma}\times [0,1]$ whose underlying $3$-manifold is $\Sigma\times [0,1]$ with boundary edges the arcs $p\times [0,1]$ for $p\in \mathcal{A}$, oriented from $0$ towards $1$. We will abusively denote $\mathbf{\Sigma}\times [0,1]$ by $\mathbf{\Sigma}$ so write $\mathcal{S}_A(\mathbf{\Sigma})$ for $\mathcal{S}_A(\mathbf{\Sigma}\times [0,1])$. Since $\mathbf{\Sigma}$ has a natural structure of algebra object in $\mathcal{M}$, then $\mathcal{S}_A(\mathbf{\Sigma})$ has a natural structure of (associative, unital) algebra where the product is defined by stacking stated tangles in the $[0,1]$ direction. 
\vspace{2mm}
\par Let $\mathbf{M}=(M, \mathcal{N})$ be a marked $3$-manifold and $a, b \in \mathcal{N}$  two (distinct) boundary edges and $D_a, D_b \subset \partial M$ two embedded discs with diameters $a$ and $b$. We denote by $\mathbf{M}_{a\# b}=(M_{a\# b}, \mathcal{N}\setminus \{a, b\})$ the marked $3$-manifold obtained by gluing $D_a$ and $D_b$ together. In \cite{LeStatedSkein, BloomquistLe} the authors introduced the \textbf{splitting morphism} 
$$ \theta_{a\# b} : \mathcal{S}_A(\mathbf{M}_{a\# b}) \to \mathcal{S}_A(\mathbf{M}).$$
When $\mathbf{M}=\mathbf{\Sigma}\times [0,1]$ is a thickened surface, $\theta_{a\#b}$ is an injective morphism of algebras. The splitting is coassociative in the sense that if $a,b,c,d$ are four pairwise distinct boundary edges, then $\theta_{a\# b}\circ \theta_{c \#d}=\theta_{c \#d} \circ  \theta_{a\# b}$. 
We refer to  \cite{LeStatedSkein, BloomquistLe}  for the precise definition and more details.
\par The \textbf{bigon} $\mathbb{B}=(D^2, \{p_L, p_R\})$ is a disc with two boundary points $p_L, p_R$, so $\mathbb{B}\times [0,1]$ is a ball with two boundary edges $a_L, a_R$. By gluing two thickened bigons $\mathbb{B}$ and $\mathbb{B}'$ by gluing $D_{a_R}$ with $D_{a'_L}$ together we get another bigon, therefore the splitting morphism $\Delta:= \theta_{a_R\# a'_L} : \mathcal{S}_A(\mathbb{B})^{\otimes 2} \to \mathcal{S}_A(\mathbb{B})$ endows $\mathcal{S}_A(\mathbb{B})$ with a bialgebra structure.
 It was proved independently in \cite{KojuQuesneyClassicalShadows} and \cite{CostantinoLe19} that $\mathcal{S}_A(\mathbb{B})$ is isomorphic to the quantum group $\mathcal{O}_q\SL_2$ where $q:= A^2$. 
 \par Let $b$ be a boundary edge of a marked $3$-manifold $\mathbf{M}=(M, \mathcal{N})$. By gluing $\mathbf{M}$ with a thickened bigon $\mathbb{B}\times [0,1]$ while gluing $D_{b}\subset \partial M$ with $D_{a_L}\subset \partial \mathbb{B}\times [0,1]$ we get another copy of $\mathbf{M}$ therefore the splitting morphism $$\Delta_b^R: =\theta_{b \# a_L} : \mathcal{S}_A(\mathbf{M}) \to \mathcal{S}_A(\mathbf{M})\otimes  \mathcal{O}_q\SL_2 $$ endows $\mathcal{S}_A(\mathbf{M})$ with a structure of right $\mathcal{O}_q\SL_2$ comodule.

\subsection{Classical limits and moduli spaces}\label{sec_charvar}

Let $M$ be a connected closed $3$-manifold. We denote by $M^*=(M\setminus \mathring{B}, \{a\})$ the marked $3$-manifold obtained by removing from $M$ the interior of a ball $B\subset M$ and placing a boundary edge $a$ in the boundary of $B$. Let $v\in a \subset M$ be a based point and denote the fundamental group $\pi_1(M,v)$ simply by $\pi_1(M)$. The \textbf{representation scheme} $R_{\SL_2}(M)$ is the affine scheme with ring of regular functions
$$ \mathcal{O}[R_{\SL_2}(M)]:= 
\quotient{\mathbb{C}[X^{\gamma}_{ij}, i,j\in \{+, -\}, \gamma \in \pi_1(M)]}{(M_{\alpha}M_{\beta}=M_{\alpha\beta}, \det(M_{\alpha})=1, \alpha, \beta\in \pi_1(M))}, $$
where $M_{\gamma}=\begin{pmatrix} X^{\gamma}_{++} & X^{\gamma}_{+-} \\ X^{\gamma}_{-+} & X^{\gamma}_{--} \end{pmatrix}$ is a $2\times 2$ matrix with coefficients in the polynomial ring $\mathbb{C}[X^{\gamma}_{ij}]$ (so the relation $M_{\alpha}M_{\beta}=M_{\alpha\beta}$ represents in fact four relations). Set $\mathcal{O}[\SL_2]:= \quotient{\mathbb{C}[x_{ij}, i,j\in \{-, +\}]}{(x_{++}x_{--}-x_{+-}x_{-+}-1)}$.
The set of characters $\Hom_{\Alg}\left(\mathcal{O}[{R}_{\SL_2}(M)], \mathbb{C}\right)$ (i.e. the closed points of $R_{\SL_2}(M)$) is in bijection with the set of representations $\rho: \pi_1(M) \to \SL_2(\mathbb{C})$. The group $\SL_2(\mathbb{C})$ acts (on the left) by conjugacy on the set of representations by the formula: 
$$ (g\cdot \rho) (\gamma):= g \rho(\gamma)g^{-1} , \quad \mbox{ for all }g\in \SL_2, \gamma \in \pi_1(M).$$
The above action is algebraic, i.e. induced by a comodule map $\Delta^R: \mathcal{O}[\mathcal{R}_{\SL_2}(M)] \to \mathcal{O}[\mathcal{R}_{\SL_2}(M)] \otimes \mathcal{O}[\SL_2]$ defined by the formula:
$$ \Delta^R (X_{ij}^{\gamma}):= \sum_{a,b= \pm}  X_{ab}^{\gamma} \otimes x_{i a}S(x_{bj}) .$$
The \textbf{character scheme} $X_{\SL_2}(M)$ is the affine scheme whose ring of regular functions $\mathcal{O}[X_{\SL_2}(M)] \subset \mathcal{O}[R_{\SL_2}(M)]$ is the subalgebra of coinvariant vectors; said differently $X_{\SL_2}(M):=R_{\SL_2}(M) \sslash \SL_2$. 
\par Let $\gamma \in \pi_1(M)$. We denote by $\tau_{\gamma}\in \mathcal{O}[X_{\SL_2}(M)]$ be the regular function defined by $\tau_{\gamma}([\rho]):=\tr(\rho(\gamma))$, i.e. $\tau_{\gamma}= X^{\gamma}_{++} + X^{\gamma}_{--}$. 
In $\mathcal{S}_{+1}(M)$ and $\mathcal{S}_{+1}(M^*)$ given two stated tangles $T_1, T_2$, the class of $T_1\cup T_2$ does not depend on how the two tangles are entangled so the formula $[T_1]\cdot [T_2]:= [T_1 \cup T_2]$ defines a product which endows $\mathcal{S}_{+1}(M)$ and $\mathcal{S}_{+1}(M^*)$ with  structures of commutative algebras. Moreover the class $[T]$ of a stated tangle $T$ does not depend on the framing of $T$. 
For $\gamma \in \pi_1(M)$, denote by $[\gamma]\in \mathcal{S}_{+1}(M)$ the class of a simple closed curve isotopic to $\gamma$. For $\varepsilon, \varepsilon' = \pm$, denote by $\gamma_{\varepsilon, \varepsilon'} \in \mathcal{S}_{+1}(M^*)$ the 
class of the stated tangle isotopic to $\gamma$ such that if $\gamma$ is oriented from an endpoint $v\in a$ towards an endpoint $w\in a$ then the state of $v$ is $\varepsilon$ and the state of $w$ is $\varepsilon'$. 

\begin{theorem}\label{theorem_classical} Let $S$ be a spin structure on $M$ with Johnson quadratic form $w_S: \mathrm{H}_1(M; \mathbb{Z}/2\mathbb{Z}) \to \mathbb{Z}/2\mathbb{Z}$. 
\begin{enumerate}
\item (Bullock \cite{Bullock}, Przytycki-Sikora \cite{PS00}) There is an algebra isomorphism $\Psi^S: \mathcal{O}[X_{\SL_2}(M)] \xrightarrow{\cong} \mathcal{S}_A(M)$ characterized by $\Psi^S(\tau_{\gamma})=(-1)^{w_S([\gamma]) +1} [\gamma]$. 
\item (K.-Murakami \cite{KojuMurakami_QCharVar}) There is an isomorphism of algebra $\Psi^S_* : \mathcal{O}[R_{\SL_2}(M)] \xrightarrow{\cong} \mathcal{S}_A(M^*)$ characterized by 
$$\Psi^S_* \begin{pmatrix} X^{\gamma}_{++} & X^{\gamma}_{+-} \\ X^{\gamma}_{-+} & X^{\gamma}_{--} \end{pmatrix} := (-1)^{w_S(\gamma)}\begin{pmatrix} 0 & -1 \\ 1 & 0 \end{pmatrix} \begin{pmatrix} \gamma_{++} & \gamma_{+-} \\ \gamma_{-+} & \gamma_{--} \end{pmatrix}.$$
Moreover $\Psi^S_*$ is a morphism of $\mathcal{O}[\SL_2]$ comodules.
\end{enumerate}
\end{theorem}

\subsection{Frobenius morphisms and coherent sheaves}

Let $\zeta \in \mathbb{C}$ be a root of unity of odd order $N\geq 3$ and denote by $\mathcal{S}_{\zeta}(\mathbf{M})$ the stated skein module at $A^{1/2}=\zeta$. 

\begin{theorem}\label{theorem_Frobenius}(\cite{BonahonWong1, LeKauffmanBracket, KojuQuesneyClassicalShadows, BloomquistLe}) Let $\mathbf{M}$ be a marked $3$-manifold. There exists a linear map $Fr: \mathcal{S}_{+1}(\mathbf{M}) \to \mathcal{S}_{\zeta}(\mathbf{M})$ whose image consists in "transparent elements". Moreover when $\mathbf{M}$ is a thickened marked surface, then $Fr$ is an injective morphism of algebras.
\end{theorem}

Theorem \ref{theorem_Frobenius} was first proved for unmarked thickened surfaces in \cite{BonahonWong1}, for unmarked $3$-manifolds in \cite{LeKauffmanBracket}, for marked thickened surfaces in \cite{KojuQuesneyClassicalShadows} and marked $3$-manifolds in \cite{BloomquistLe}. We call \textbf{Frobenius morphism} the linear map $Fr$. It is roughly described as follows, we refer to \cite{BloomquistLe} for details. 
The $n$\textbf{-th Chebyshev polynomial of first kind} is the polynomial $T_n(X) \in \mathbb{Z}[X]$ defined recursively by $T_0=1$, $T_1=X$ and $T_{n-1}(X)+T_{n+1}(X)=XT_n(X)$. For $T$ a stated tangle in $\mathbf{M}$, then $Fr([T])$ is defined by replacing each open component of $T$, i.e. each stated arc $\alpha_{ij} \subset T$, by $N$ parallel copies of $\alpha_{i,j}$ pushed in the framing direction, and replacing each closed component of $T$, i.e. each closed curve $\gamma \subset T$, by $T_N(\gamma)$. The transparency condition means that if $T_1, T_2$ are two stated arcs, then the class $[Fr(T_1)\cup T_2] \in \mathcal{S}_{\zeta}(\mathbf{M})$ does not depend on how $Fr(T_1)$ and $T_2$ are entangled so the assignment $[T_1] \rhd [T_2]:= [Fr(T_1) \cup T_2]$ defines a module structure $\rhd: \mathcal{S}_{+1}(\mathbf{M}) \otimes \mathcal{S}_{\zeta}(\mathbf{M}) \to \mathcal{S}_{\zeta}(\mathbf{M})$, so endows $\mathcal{S}_{\zeta}(\mathbf{M})$ with a structure of $\mathcal{S}_{+1}(\mathbf{M})$ module. 
\par Now suppose that $M$ is a closed connected $3$-manifold and denote the two Frobenius morphisms by $Fr: \mathcal{S}_{+1}(M) \to \mathcal{S}_{\zeta}(M)$ and $Fr_*: \mathcal{S}_{+1}(M^*) \to \mathcal{S}_{\zeta}(M^*)$. Fix a spin structure $S$ on $M$. Using the isomorphisms $\Psi^S$ and $\Psi^S_*$ of Theorem \ref{theorem_classical}, $\mathcal{S}_{\zeta}(M)$ and $\mathcal{S}_{\zeta}(M^*)$ acquire structures of $\mathcal{O}[X_{\SL_2}(M)]$ and $\mathcal{O}[R_{\SL_2}(M)]$ modules respectively, so they define coherent sheaves over $X_{\SL_2}(M)$ and $R_{\SL_2}(M)$.

\begin{definition} We denote by $\mathscr{L}^M \to X_{\SL_2}(M)$ and $\mathscr{I}^M \to R_{\SL_2}(M)$ the coherent sheaves defined by the Frobenius morphism $Fr$ and $Fr_*$ respectively.
\end{definition}

Let $\pi: R_{\SL_2}(M) \to X_{\SL_2}(M)$ be the projection map coming from the inclusion $\mathcal{O}[X_{\SL_2}(M)] \subset \mathcal{O}[R_{\SL_2}(M)]$. Let $\iota : \mathcal{S}_A(M) \to \mathcal{S}_A(M^*)$ be the natural inclusion morphism. Then $\iota$ defines a sheaf morphism $\pi': \pi^* \mathscr{L}^M \to \mathscr{I}^M$ over $R_{\SL_2}(M)$. 
\par Let $\rho\in R_{\SL_2}(M)$ be a closed point corresponding to a maximal ideal $\mathfrak{m}_{\rho} \in \mathcal{S}_{+1}(M)$ through $\Psi^S$. Similarly denote by $\mathfrak{m}_{[\rho]}\subset \mathcal{S}_{+1}(M)$ the maximal ideal corresponding to $[\rho] \in X_{\SL_2}(M)$. The \textbf{skein module reduced at }$[\rho]$ and the \textbf{stated skein module reduced at }$\rho$ are the quotients: 
$$ \mathcal{S}_{\zeta, [\rho]}(M):= \quotient{\mathcal{S}_{\zeta}(M)}{Fr(\mathfrak{m}_{[\rho]})\mathcal{S}_{\zeta}(M)} \quad \mbox{ and }\quad  \mathcal{S}_{\zeta, \rho}(M^*):= \quotient{\mathcal{S}_{\zeta}(M^*)}{Fr_*(\mathfrak{m}_{\rho})\mathcal{S}_{\zeta}(M*)}.$$

\subsection{Heegaard splittings}

For $M$ a connected closed $3$-manifold, a \textbf{Heegaard splitting} is the data of an embedding $\Sigma_g\times [0,1]\hookrightarrow M$ such that $M\setminus \Sigma_g\times [0,1]$ has two connected components whose closures in $M$ are two handlebodies $H_g^1$ and $H_g^2$, so $M\cong H_g^1 \cup_{\Sigma_g}H_g^2$. We suppose that $H_g^{1}$ intersects $\Sigma\times \{0\}$ while $H_g^{2}$ intersects $\Sigma \times \{1\}$.
 Let $A^{1/2}$ be an invertible element in a commutative unital ring $k$. Identifying $(\Sigma\times[0,1]) \cup_{\Sigma\times \{1\}} H_g^2$ with $H_g^2$ and using the functoriality of $\mathcal{S}_A$, $\mathcal{S}_A(H^2_g)$ acquires a structure of $\mathcal{S}_A(\Sigma_g)$ left module. Similarly, identifying $H_g^1 \cup_{\Sigma \times \{0\}} (\Sigma \times [0,1])$ with $H_g^1$,  $\mathcal{S}_A(H^1_g)$ acquires a structure of $\mathcal{S}_A(\Sigma_g)$ right module.
 The embeddings $H_g^1 \hookrightarrow M$ and ${H}^2_g \hookrightarrow M$ induce a linear map
  $$\varphi : \mathcal{S}_A(H^1_g) \otimes_{\mathcal{S}_A(\Sigma_g)} \mathcal{S}_A(H^2_g) \to \mathcal{S}_A(M).$$
\par Choose a ball $B\subset M$  which intersects $\Sigma_g \times [0,1] \subset M$  along a thickened disc $D=\times [0,1] B \cap (\Sigma_g\times [0,1])$ such that  $B\setminus D\times [0,1]$ consists in two hemispheres $E_1\subset H_g^1$ and $E_2 \subset H_g^2$. 
 Let $a\subset \partial B$ be an oriented arc such that $(1)$ $a$ intersects $D\times [0,1]$  along an arc $\{v\} \times [0,1]$ with $v\in \partial D$ oriented from $(v,0)$ towards $(v, 1)$ and $(2)$ for $i=1,2$, $a_i:=a\cap E_i$ is a connected arc. Consider the marked $3$-manifolds $M^*=(M\setminus \mathring{B}, \{a\})$, $H_g^{*, i}:= ( \overline{H_g^i \setminus E_i} ,  \{a_i\})$ for $i=1,2$  and the marked surface $\Sigma_g^*=(\Sigma_g \setminus \mathring{D}, \{v\})$. Like before, $\mathcal{S}_A(H_g^{*, 1})$ and $\mathcal{S}_A({H}_g^{*, 2})$ acquire structures of right and left $\mathcal{S}_A(\Sigma_g^*)$ modules and we have a linear morphism 
$$\varphi_* : \mathcal{S}_A(H_g^{*, 1}) \otimes_{\mathcal{S}_A(\Sigma_g^*)} \mathcal{S}_A(H_g^{*, 2}) \to \mathcal{S}_A(M^*).$$

\begin{theorem}\label{theorem_heegaard}
\begin{enumerate}
\item (\cite{Przytycki_skein}) $\varphi : \mathcal{S}_A(H^1_g) \otimes_{\mathcal{S}_A(\Sigma_g)} \mathcal{S}_A(H^2_g) \to \mathcal{S}_A(M)$ is a linear isomorphism.
\item (\cite{GunninghamJordanSafranov_FinitenessConjecture, KojuMurakami_QCharVar, CostantinoLe_SSkeinTQFT}) $\varphi_* : \mathcal{S}_A(H_g^{*, 1}) \otimes_{\mathcal{S}_A(\Sigma_g^*)} \mathcal{S}_A(H_g^{*, 2}) \to \mathcal{S}_A(M^*)$ is a linear isomorphism.
\end{enumerate}
\end{theorem}

More precisely, the fact that $\varphi$ is an isomorphism follows from a classical argument of Hoste and Pryszyticki detailed in \cite{Przytycki_skein} (see also \cite{GunninghamJordanSafranov_FinitenessConjecture}). The fact that $\varphi_*$ is an isomorphism is somehow implicit in \cite{GunninghamJordanSafranov_FinitenessConjecture} using a very different language (no mention of $M^*$ is present here) and was reproved independently in \cite[Corollary $6.14$]{KojuMurakami_QCharVar} and \cite[Theorem $6.5$]{CostantinoLe_SSkeinTQFT}. At the classical level,  the Van Kampen theorem asserts that we have a push-out square of groups
 $$ \begin{tikzcd}
  \pi_1(\Sigma_g, v)
   \ar[r, ] \ar[d, ] 
   & 
   \pi_1(H_g^{1}, v) \ar[d] \\
  \pi_1(H_g^{2},v ) \ar[r] & \pi_1(M, v)
  \end{tikzcd}
  $$
which induces isomorphisms 
$$  \mathcal{O}[{R}_{\SL_2}(M)]\cong  \mathcal{O}[{R}_{\SL_2}(H_g^{1})]\otimes_{ \mathcal{O}[{R}_{\SL_2}(\Sigma_g)]}  \mathcal{O}[{R}_{\SL_2}(H_g^{2})] \quad \mbox{ and }\quad $$
$$
\mathcal{O}[{X}_{\SL_2}(M)]\cong  \mathcal{O}[{X}_{\SL_2}(H_g^{(1)})]\otimes_{ \mathcal{O}[{X}_{\SL_2}(\Sigma_g)]}  \mathcal{O}[{X}_{\SL_2}(H_g^{(2)})]$$
as proved, for instance, in \cite[Proposition $2.10$]{MarcheCharVarSkein}. Under these isomorphisms, the isomorphism $\varphi$ and $\varphi_*$ are $\mathcal{O}[X_{\SL_2}(M)]$ and $\mathcal{O}[R_{\SL_2}(M)]$ equivariant respectively. 
Therefore, Theorem \ref{theorem_heegaard}, implies the 

\begin{corollary}\label{coro_heegaard}
Let $\zeta$ be a root of unity of odd order.
Let $\rho : \pi_1(M, v) \to \SL_2$ be a representation and denote by $\rho_1\in R_{\SL_2}({H}^1_g)$, $\rho_2\in R_{\SL_2}(H^2_g)$ and $\rho_{\Sigma} \in R_{\SL_2}(\Sigma_g)$ the representations obtained by precomposing with the inclusion morphisms $\pi_1(H^1_g)\to \pi_1(M)$, $\pi_1({H}^2_g)\to \pi_1(M)$ and $\pi_1(\Sigma_g) \to \pi_1(M)$. Then $\varphi$ and $\overline{\varphi}$ induce isomorphisms
$$ \varphi_{[\rho]}:  \mathcal{S}_{\zeta, [\rho_1]}(H^1_g) \otimes_{\mathcal{S}_{\zeta, [\rho_{\Sigma}]}(\Sigma_g)} \mathcal{S}_{\zeta, [\rho_2]}(H^2_g) \to \mathcal{S}_{\zeta, [\rho]}(M)
 \quad
\varphi_{\rho, *}:  \mathcal{S}_{\zeta, \rho_1}(H_g^{*, 1}) \otimes_{\mathcal{S}_{\zeta, \rho_{\Sigma}}(\Sigma_g^*)} \mathcal{S}_{\zeta, \rho_2}(H_g^{*, 2}) \to \mathcal{S}_{\zeta, \rho}(M^*).
$$
\end{corollary}

\subsection{The inclusion of skein into stated skein is injective}

\begin{proposition}\label{prop_injection} Let $k$ be a unital commutative ring and $A^{1/2} \in k^{\times}$. 
 The inclusion $\iota: \mathcal{S}_A(M) \to \mathcal{S}_A(M^*)$ is injective.
 \end{proposition}

 For $\mathbf{\Sigma}= (\Sigma, \mathcal{A})$ a marked surface, the algebra $\mathcal{S}_A(\mathbf{\Sigma})$ admits the following filtration introduced in \cite[Section $2.11$]{LeStatedSkein}. Let $\mathcal{F}_n := \Span_k( [T,s], | \partial T | \leq 2n )\subset \mathcal{S}_A(\mathbf{\Sigma})$ denote the span of classes of stated diagrams with at most $2n$ endpoints. Then $\mathcal{S}_A(\mathbf{\Sigma}) = \cup_{n\geq 0} \mathcal{F}_n$, $\mathcal{F}_n \subset \mathcal{F}_{n+1}$  and $\mathcal{F}_n \cdot \mathcal{F}_m \subset \mathcal{F}_{n+m}$. Note that $\mathcal{F}_0$ is the image of the usual skein algebra by the natural inclusion $\mathcal{S}_A(\Sigma) \hookrightarrow \mathcal{S}_A(\mathbf{\Sigma})$ which is known to be  injective for marked surfaces. Let $Gr_\bullet= \oplus_n Gr_n$ be the associated graded algebra where $Gr_0=\mathcal{F}_0$ and $Gr_n= \quotient{\mathcal{F}_n}{\mathcal{F}_{n-1}}$ for $n\geq 1$. It is shown in \cite{LeStatedSkein} that the quotient map $\pi_{\mathbf{\Sigma}}: \mathcal{S}_A(\mathbf{\Sigma}) \to Gr_{\bullet}$ is a linear isomorphism.
 
 \begin{proof}[Proof of Proposition \ref{prop_injection}]
 Consider a Heegaard splitting $M= H^1_g \cup_{\Sigma_g} H^2_g$. Write $\mathcal{A}:= \mathcal{S}_A(\Sigma_g^*)$, $M_1:= \mathcal{S}_A(H_g^{*, 1})$, $M_2:= \mathcal{S}_A(H_g^{*, 2})$ and consider the filtrations $\mathcal{A}= \cup_{n\geq 0} \mathcal{A}^{n}$, $M_1= \cup_{n\geq 0} M_1^{n}$ and $M_2= \cup_{n\geq 0} M_2^n$ as defined above where we consider $H_g^1, H_g^2$ as thickened holed discs. Recall that for thickened marked surface, the inclusion of skein into stated skein injective so $\mathcal{A}^0 = \mathcal{S}_A(\Sigma_g)$ and $M_i^0 = \mathcal{S}_A(H_g^i)$ for $i=1,2$. 
  By Theorem \ref{theorem_heegaard}, we have isomorphisms
 $$ \mathcal{S}_A(M) \cong M_1^0 \otimes_{\mathcal{A}^0} M_2^0 \quad \mbox{ and } \quad \mathcal{S}_A(M^*) \cong M_1 \otimes_\mathcal{A} M_2, $$
 so we need to prove that the inclusion $M_1^0 \otimes_k M_2^0 \subset M_1 \otimes_k M_2$ induces an injective morphism $M_1^0 \otimes_{\mathcal{A}^0} M_2^0 \to M_1 \otimes_\mathcal{A} M_2$. Writing 
 \begin{multline*} J:= \Span_k \left( m_1a \otimes m_2 - m_1 \otimes a m_2, \quad m_i \in M_i^0, a \in \mathcal{A} \right)\subset M_1\otimes_k M_2  \quad \mbox{ and } \\ 
 \quad J_0:= \Span_k \left( m_1a \otimes m_2 - m_1 \otimes a m_2, \quad m_i \in M_i^0, a \in \mathcal{A}^0 \right)\subset M_1^0\otimes_k M_2^0 , \end{multline*}
 this amounts to prove that $J \cap (M_1^0 \otimes_k M_2^0)= J_0$. The inclusion $J_0 \subset J \cap (M_1^0 \otimes_k M_2^0)$ is obvious.
 If $v_1 \in M_1$, $v_2 \in M_2$ represent the classes of the empty links, then $v_1, v_2$ are cyclic vectors, i.e. $M_1= v_1 \cdot A$ and $M_2= A \cdot v_2$ so if $I_1:= \ker(A \to M_1, a\mapsto v_1 \cdot a)$ and $I_2:= \ker(A \to M_2, a\mapsto a\cdot v_2)$ then as right and left modules we have $M_1\cong I_1\backslash {\raisebox{.2em}{$\mathcal{A}$}} $ and $M_2 \cong \quotient{\mathcal{A}}{I_2}$. By definition of the filtrations, $M_1^n \cong (I_1\cap \mathcal{A}^n) \backslash {\raisebox{.2em}{$\mathcal{A}^n$}}$ and $M_2^n \cong \quotient{\mathcal{A}^n}{(I_2\cap \mathcal{A}^n)}$. So if $a \in \mathcal{A}$ and $m_2 \in M_2^0$ are such that $am_2 \in M_2^0$ then $a=a_0 +a_2$ with $a_0 \in \mathcal{A}^0$ and $a_2 \in I_2$. Similarly if $a \in \mathcal{A}$ and $m_1 \in M_1^0$ are such that $m_1 a  \in M_1^0$ then $a=a_0 +a_1$ with $a_0 \in \mathcal{A}^0$ and $a_1 \in I_1$. This implies that if $a \in \mathcal{A}$, $m_1\in M_1^0$ and $m_2 \in M_2^0$ are such that $m_1a\otimes m_2 - m_1\otimes am_2 \in M_1^0\otimes_k M_2^0$ then $a=a_0 + a_1 + a_2$ with $a_0 \in \mathcal{A}^0$, $a_1\in I_1$ and $a_2 \in I_2$ so 
 $$ m_1a\otimes m_2 - m_1\otimes am_2 = m_1a_0 \otimes m_2 - m_1 \otimes a_0 m_2 \in J_0.$$
 This proves the inclusion $J \cap (M_1^0 \otimes_k M_2^0)\subset J_0$. So the morphism $M_1^0 \otimes_{A^0} M_2^0 \to M_1 \otimes_A M_2$ is injective and so is $\iota: \mathcal{S}_A(M) \to \mathcal{S}_A(M^*)$.
 \end{proof}
 
\subsection{The Reynolds operator}

Recall from Section \ref{sec_charvar} that $\mathcal{O}[R_{\SL_2}(M)]$ has a structure of $\mathcal{O}[\SL_2]$ comodule. Dually it admits a structure of $\SL_2$-module where $\mathcal{O}[X_{\SL_2}(M)]:=\mathcal{O}[R_{\SL_2}(M)]^{\SL_2}$ is the subset of invariant vectors. 
Said differently, the vector space $\mathcal{O}[R_{\SL_2}(M)]$ decomposes as a direct sum $\oplus_i n_i V_i$ over the simple $\SL_2$ modules $V_i$ and  $\mathcal{O}[X_{\SL_2}(M)]=n_0V_0$ is the factor corresponding to the trivial representation. The \textit{Reynolds operator} is the projection $\mathcal{R}: \mathcal{O}[R_{\SL_2}(M)] \to \mathcal{O}[X_{\SL_2}(M)]$ orthogonal to $\oplus_{i>0}n_i V_i$, i.e. the unique $\SL_2$ equivariant projection. Alternatively, it is commonly defined by the formula
$$ \mathcal{R}(f) = \int_{g \in \SU_2} g f(g^{-1}\bullet)d\mu(g), \quad \mbox{ for }f\in \mathcal{O}[R_{\SL_2}(M)], $$
where $d\mu$ represents the Haar measure.
In particular, it satisfies $\mathcal{R}(f_0 g) = f_0 \mathcal{R}(g)$ when $f_0\in \mathcal{O}[X_{\SL_2}(M)]$. The existence of the Reynolds operator implies the following classical lemma for which we recall the proof for the reader conveniance.

\begin{lemma}\label{lemma_Reynolds} For $I\subset  \mathcal{O}[X_{\SL_2}(M)]$ an ideal, we have $I  \mathcal{O}[R_{\SL_2}(M)] \cap  \mathcal{O}[X_{\SL_2}(M)] = I$. \end{lemma}

\begin{proof} If $f\in  I  \mathcal{O}[R_{\SL_2}(M)] $ then $f= \sum_i f_i h_i$ with $f_i \in I$ and $h_i \in \mathcal{O}[X_{\SL_2}(M)]$. Moreover if $f \in  \mathcal{O}[X_{\SL_2}(M)]$, then 
$$ f = \mathcal{R}(f) = \sum_i \mathcal{R}(f_i)h_i \in I.$$
This proves the inclusion $I  \mathcal{O}[R_{\SL_2}(M)] \cap  \mathcal{O}[X_{\SL_2}(M)] \subset I$. The reverse inclusion is obvious.
\end{proof}

This lemma implies the following.
\begin{lemma}\label{lemma_Reynolds2} Let $\rho \in R_{\SL_2}(M)$ and $[\rho]=\pi(\rho) \in X_{\SL_2}(M)$. Then the inclusion morphism $\iota : \mathcal{S}_{\zeta}(M) \hookrightarrow \mathcal{S}_{\zeta}(M^*)$ induces an injective morphism $\restriction{\iota}{\rho}: \mathcal{S}_{\zeta, [\rho]}(M) \hookrightarrow \mathcal{S}_{\zeta, \rho}(M^*)$. 
\end{lemma}

\begin{proof}
Let $\mathfrak{m}_{\rho}\subset \mathcal{S}_{+1}(M^*)$ and $\mathfrak{m}_{[\rho]} \subset \mathcal{S}_{+1}(M)$ be the maximal ideals corresponding to $\rho$ and $[\rho]$ respectively and consider the morphism $i: \mathcal{S}_{+1}(M)\hookrightarrow \mathcal{S}_{+1}(M^*)$ corresponding to  the morphism $\iota$ at $A=+1$. Using the identifications $\mathcal{S}_{+1}(M) \cong \mathcal{O}[X_{\SL_2}(M)]$ and $\mathcal{S}_{+1}(M^*) \cong \mathcal{O}[R_{\SL_2}(M)]$, Lemma \ref{lemma_Reynolds} implies the  equality $i^{-1}(\mathfrak{m}_{\rho})=\mathfrak{m}_{[\rho]}$. We thus  have the inclusion $\iota^{-1}(Fr_*(\mathfrak{m}_{\rho})\mathcal{S}_{\zeta}(M^*)) \subset Fr(\mathfrak{m}_{[\rho]}) \mathcal{S}_{\zeta}(M)$ so $\iota$ induces an injective morphism 
$$\restriction{\iota}{\rho}: \mathcal{S}_{\zeta, [\rho]}(M)=\quotient{\mathcal{S}_{\zeta}(M)}{Fr(\mathfrak{m}_{[\rho]}) \mathcal{S}_{\zeta}(M)}  \hookrightarrow \quotient{\mathcal{S}_{\zeta}(M^*)}{Fr(\mathfrak{m}_{\rho}) \mathcal{S}_{\zeta}(M^*)}=  \mathcal{S}_{\zeta, \rho}(M^*).$$
\end{proof}

\section{Azumaya loci of stated skein algebras}

A complex algebra $\mathscr{A}$ is called \textbf{almost Azumaya} if $(1)$ it is prime, $(2)$ it is affine (finitely generated as an algebra) and $(3)$ it is finitely generated as a module over its center $Z\subset \mathscr{A}$. Let $Q(Z)$ be the fraction field of $Z$ (obtained by localizing by every non zero element). Then $\mathscr{A}\otimes_Z Q(Z)$ is a central simple algebra with center $Q(Z)$, so there exists a finite field extension $Q(Z)\subset F$ such that $\mathscr{A}\otimes_Z F\cong \Mat_D(F)$ is a matrix algebra. The size $D$ is called the \textbf{PI-degree} of $\mathscr{A}$ and it is characterized by the formula $\dim_{Q(Z)} \mathscr{A}\otimes_Z Q(Z) = D^2$. 
\par Let $X:= \Specm(Z)$. For $x\in X$, we denote by $\mathfrak{m}_x \subset Z$ the corresponding maximal ideal and write $\mathscr{A}_x:= \quotient{\mathscr{A}}{\mathfrak{m}_x \mathscr{A}}$. The \textbf{Azumaya locus} of $\mathscr{A}$ is the subset: 
$$ \mathcal{AL} := \{ x \in X \mbox{ such that } \mathscr{A}_x \cong \Mat_D(\mathbb{C}) \}.$$
By \cite[Theorem III.I.7]{BrownGoodearl}, the Azumaya locus is an open dense subset. Let $\alpha_{\partial} \in \pi_1(\Sigma_{g,1}, v)$ be the class of a small loop encircling the boundary component once, so if $\alpha_1, \beta_1, \ldots, \alpha_g, \beta_g \in \pi_1(\Sigma_{g,1})$ are meridians and longitudes which freely generate $\pi_1(\Sigma_{g,1})$, then $\alpha_{\partial}=[\alpha_1, \beta_1] \ldots [\alpha_g, \beta_g]$. 
 The \textit{moment map} is the regular map  $\mu : R_{\SL_2}(\Sigma_{g,1}) \to \SL_2(\mathbb{C})$ defined by $\mu(\rho):= \rho(\alpha_{\partial})$. Recall the (simple) Bruhat cell decomposition $\SL_2(\mathbb{C})= \SL_2^{(0)} \sqcup \SL_2^{(1)}$ where $\SL_2^{(0)}=\left\{ \begin{pmatrix} a & b \\ c & d \end{pmatrix} \mbox{ such that } a\neq 0 \right\}$. 

\begin{theorem}\label{theorem_Azumaya} Let $\zeta$ be a root of unity of odd order $N\geq 3$. 
\begin{enumerate}
\item (\cite{FrohmanKaniaLe_UnicityRep, FrohmanKaniaLe_DimSkein}) $\mathcal{S}_{\zeta}(\Sigma_g)$ is almost Azumaya of PI-degree $D_{\Sigma_g}= N^{3g-3}$ if $g\geq 2$ and $D_{\Sigma_1}=N$ if $g=1$. Moreover $Fr: \mathcal{S}_{+1}(\Sigma_g) \to Z(\mathcal{S}_{\zeta}(\Sigma_g))$ is an isomorphism, so $\Specm(Z(\mathcal{S}_{\zeta}(\Sigma_g))) \cong X_{\SL_2}(\Sigma_g)$. 
\item (\cite{GanevJordanSafranov_FrobeniusMorphism, KojuKaruo_Azumaya}) The Azumaya locus of  $\mathcal{S}_{\zeta}(\Sigma_g)$ is the set $\mathcal{AL}(\Sigma_g)\subset X_{\SL_2}(\Sigma_g)$ of classes of non central representations.
\item (\cite{KojuAzumayaSkein}) $\mathcal{S}_{\zeta}(\Sigma_g^*)$ is almost Azumaya. 
\item (\cite{GanevJordanSafranov_FrobeniusMorphism}, see also \cite{KojuMCGRepQT, BaseilhacFaitgRoche_LGFT3}) $Fr_*: \mathcal{S}_{+1}(\Sigma_g^*) \to Z(\mathcal{S}_{\zeta}(\Sigma_g^*))$ is an isomorphism so $\Specm(Z(\mathcal{S}_{\zeta}(\Sigma_g^*)))\cong R_{\SL_2}(\Sigma_{g,1})$. Moreover $\mathcal{S}_{\zeta}(\Sigma_g^*)$ has  PI-degree $D_{\Sigma_g^*}= N^{3g}$ and its Azumaya locus is the subset $\mathcal{AL}(\Sigma_g^*) := \mu^{-1}(\SL_2^{(0)}) \subset R_{\SL_2}(\Sigma_{g,1})$.
\end{enumerate}
\end{theorem}
Note that $R_{\SL_2}(\Sigma_g)$ identifies with the subset $\mu^{-1}(\mathds{1}_2) \subset R_{\SL_2}(\Sigma_{g,1})$. Under this identification, since $\mathds{1}_2\in \SL_2^{(0)}$, then $R_{\SL_2}(\Sigma_g)$ is included in the Azumaya locus of $\mathcal{S}_{\zeta}(\Sigma_g^*)$.

\section{Proof of the main theorems}
In all this section, we fix $\zeta$ a root of unity of odd order $N\geq 3$ and $M$ a connected closed $3$ manifold equipped with a genus $g$ Heegaard splitting $M\cong {H}^1_g \cup_{\Sigma_g} H^2_g$. 
The following is Exercise $II.5.8$ in Hartshorne's book \cite{Hart}.

\begin{lemma}\label{lemma_hartshorne} Let $X$ be a Noetherian reduced scheme over $\mathbb{C}$. Let $\mathcal{F}\to X$ be a coherent sheaf such that for all $x\in X$, then $\mathcal{F}_{| x}:= \mathcal{F} \otimes_{\mathcal{O}_X} \mathbb{C}_x$ is one dimensional. Then $\mathcal{F}\to X$ is a line bundle.
\end{lemma}

\begin{proof}
Let $x\in X$ and let us prove that $\mathcal{F}$ if free of rank one in a neighborhood of $x$. Since the statement is local and $\mathcal{F}$ is coherent, we can suppose that $X=\Spec(A)$ is affine and that $\mathcal{F}=\widetilde{M}$ is associated to an $A$ module $M$. Let $\mathfrak{p} \subset A$ be the prime ideal corresponding to $x$.
Let $\mathcal{F}_x$ denote the stalk of $\mathcal{F}$ at $x$. Then $\mathcal{F}_x \otimes_{\mathcal{O}_{X,x}} \mathbb{C}_x \cong \mathcal{F}_{| x}$ is one dimensional so $\mathcal{F}_x$ is a cyclic $\mathcal{O}_{X,x}$ module. By Nakayama lemma, the same is true in a neighborhood of $x$. More precisely, there exists a generator $n$ of $\mathcal{F}_x$ of the form $n=\sum_i \frac{a_i}{s_i}m_i$ with $m_i \in \mathcal{F}$ and $a_i \in A$, $s_i \in A\setminus \mathfrak{p}$. Let $s:=\prod_i s_i$ and consider the neighborhood $D(s)$ of $x$. We have then an exact sequence:
$$ 0 \to \ker{\varphi} \to A_s \xrightarrow{\varphi} M_s \to 0, $$
which also holds when localized in any $\mathfrak{q} \in D(s)$. By assumption,  for all $\mathfrak{q} \in D(s)$ we have $\dim(M_{\mathfrak{q}}\otimes \mathbb{C}_{\mathfrak{q}})=1$ so $\dim(\ker{\varphi}_{\mathfrak{q}}\otimes \mathbb{C}_{\mathfrak{q}})=0$. So every element of $\ker{\varphi}$ is a sum of elements of $\mathfrak{q}A_s$ with $\mathfrak{q} \in D(s)$, so is in the nilradical of $A_s$. Since $X$ is reduced, so is $A_s$ thus $\ker{\varphi}=0$ and $M_s$ is free of rank one over $A_s$.
\end{proof}

\begin{lemma}\label{lemma_Hg} Let $\rho: \pi_1(H_g) \to \SL_2(\mathbb{C})$ be a representation. Then $\mathcal{S}_{\zeta, \rho}(H_g^*)$ has dimension $N^{3g}$. 
\end{lemma}

\begin{proof} This follows from the fact, proved in \cite[Lemma $8.9$]{KojuKaruo_RepRSSkein}, that $\mathcal{S}_{\zeta}(H_g^*)$ is free over $Z^0_{H_g^*}$ of rank $N^{3g}$.
\end{proof}

\begin{theorem}\label{theorem_Ibundle} For every $\rho \in R_{\SL_2}(M)$ then $\dim(\mathcal{S}_{\zeta, \rho}(M^*))=1$.  If $R_{\SL_2}(M)$ is reduced, the sheaf $\mathscr{I}^M \to R_{\SL_2}(M)$ is a line bundle.
\end{theorem}

\begin{proof} Let $\rho: \pi_1(M) \to \SL_2$ be a representation and denote by $\rho_1, \rho_2, \rho_{\Sigma}$ the induced representations of $\pi_1({H}^1_g), \pi_1(H^2_g)$ and $\pi_1(\Sigma_{g})$ respectively as in Corollary \ref{coro_heegaard}. 
Here we consider $\rho_{\Sigma}$ as a representation of $\pi_1(\Sigma_{g,1})$ as well through the inclusion $\pi_1(\Sigma_g) \hookrightarrow \pi_1(\Sigma_{g,1})$.
In particular $\rho_{\Sigma}(\alpha_{\partial})=\mathds{1}_2\in \SL_2^{(0)}$ so, by Theorem \ref{theorem_Azumaya}, $\rho_{\Sigma}$ belongs to the Azumaya locus of $\mathcal{S}_{\zeta}(\Sigma_g^*)$ so $\mathcal{S}_{\zeta, \rho_{\Sigma}}(\Sigma_g^*)\cong \Mat_{D}(\mathbb{C})$ where $D=N^{3g}$. By Lemma \ref{lemma_Hg}, the left and right modules $\mathcal{S}_{\zeta, \rho_2}(H_g^{*, 2})$ and $\mathcal{S}_{\zeta, \rho_1}({H}_g^{*, 1})$ have dimension $D$. Since the only left (resp. right) non trivial $\Mat_{D}(\mathbb{C})$ module of dimension $D$ is the standard left (resp. right) module $\mathbb{C}^D$, Corollary \ref{coro_heegaard} implies 
$$\mathcal{S}_{\zeta, \rho}(M^*) \cong \mathcal{S}_{\zeta, \rho_1}(H_g^{*, 1}) \otimes_{\mathcal{S}_{\zeta, \rho_{\Sigma}}(\Sigma_g^*)} \mathcal{S}_{\zeta, \rho_2}(H_g^{*, 2}) \cong \mathbb{C}^D\otimes_{\Mat_D(\mathbb{C})} \mathbb{C}^D \cong \mathbb{C}.$$
So every quotient  $\restriction{\mathscr{I}^M}{\rho}= \mathcal{S}_{\zeta, \rho}(M^*)$ is one dimensional and $\mathscr{I}^M$ is a line bundle by Lemma \ref{lemma_hartshorne}.

\end{proof}

Let $[\rho]\in X_{\SL_2}(M)$ be a non central representation and let $D$ be the PI-degree of $\mathcal{S}_{\zeta}(\Sigma_g)$. By Theorem \ref{theorem_Azumaya}, $[\rho_{\Sigma}]$ belongs to the Azumaya locus of $\mathcal{S}_{\zeta}(\Sigma_g)$. Moreover, it is proved in \cite{FKL_GeometricSkein, KojuKaruo_Azumaya} that $\mathcal{S}_A(H_g)$ has dimension $D$. So using Theorem \ref{theorem_Azumaya}, like in the previous proof,  one has 
$$\mathcal{S}_{\zeta, [\rho]}(M) \cong \mathcal{S}_{\zeta, [\rho_1]}(H_g) \otimes_{\mathcal{S}_{\zeta, [\rho_{\Sigma}]}(\Sigma_g)} \mathcal{S}_{\zeta, [\rho_2]}(H_g) \cong \mathbb{C}^D\otimes_{\Mat_D(\mathbb{C})} \mathbb{C}^D \cong \mathbb{C}.$$

We thus get the 

\begin{theorem}\label{theorem_FKLKK}(\cite{FKL_GeometricSkein, KojuKaruo_Azumaya}) For $[\rho]$ the class of either an irreducible or a diagonal representation, then $\dim(\mathcal{S}_{\zeta, [\rho]}(M))=1$.  So if $X_{\SL_2}(M)$ is reduced, the restriction of $\mathscr{L}^M \to X_{\SL_2}(M)$ to the locus of non central representations is a line bundle.
\end{theorem}

When $\rho$ is a central representation, then $[\rho_{\Sigma}]$ does not belong to the Azumaya locus of $\mathcal{S}_{\zeta}(\Sigma_g)$ and the module $\mathcal{S}_{\zeta, [\rho]}(H_g)$ is not well understood (though it contains a well-understood simple quotient which is the space $V(\Sigma_g)$ of the Witten-Reshetikhin-Turaev TQFT) so the previous strategy is hard to handle for central representations.

\begin{lemma}\label{lemma_preliminary} If $\rho, \rho'$ are two central representations, then $\mathcal{S}_{\zeta, [\rho]}(M)\cong \mathcal{S}_{\zeta, [\rho']}(M)$ and $\dim(\mathcal{S}_{\zeta, [\rho]}(M))\geq 1$. 
\end{lemma}

\begin{proof} The skein module $\mathcal{S}_{\zeta}(M)$ admits a natural $\mathrm{H}^1(M; \mathbb{Z}/2\mathbb{Z})$ action defined for a class $\chi\in  \mathrm{H}^1(M; \mathbb{Z}/2\mathbb{Z})$ and a link $L=L_1\cup \ldots \cup L_n \subset M$ by $\chi \cdot [L]= (-1)^{\sum_i \chi([L_i])} [L]$. The map $Fr \circ \Psi^S : \mathcal{O}[X_{\SL_2}(M)] \to \mathcal{S}_{\zeta}(M)$ is clearly  $\mathrm{H}^1(M; \mathbb{Z}/2\mathbb{Z})$ equivariant, where $X_{\SL_2}(M)$ is equipped with the obvious action sending $[\rho]$ to the class of $\chi\cdot \rho : \gamma \mapsto (-1)^{\chi([\gamma])} \rho(\gamma)$. Since the action of $\mathrm{H}^1(M; \mathbb{Z}/2\mathbb{Z})$ on central representations is transitive, for $\rho$, $\rho'$ two central representations, we can find $\chi\in \mathrm{H}^1(M; \mathbb{Z}/2\mathbb{Z})$ with $\rho'=\chi\cdot \rho$ and the linear isomorphism $\chi\cdot \bullet : \mathcal{S}_{\zeta}(M) \to \mathcal{S}_{\zeta}(M)$ induces an isomorphism $\mathcal{S}_{\zeta, [\rho]}(M)\cong \mathcal{S}_{\zeta, [\rho']}(M)$.
\par When $\rho$ is central, the Witten-Reshetikhin-Turaev invariant defines a linear map $WRT_{\rho, \zeta} : \mathcal{S}_{\zeta, [\rho]}(M)\to \mathbb{C}$ which were proved to be surjective in \cite[Proposition $2.4$]{DetcherryKalfagianniSikora_SkeinDim}. Therefore $\mathcal{S}_{\zeta, [\rho]}(M)\neq 0$.

\end{proof}




\begin{theorem}\label{theorem_isom} 
If $X_{\SL_2}(M)$ is reduced, then  $\mathscr{L}^M \to X_{\SL_2}(M)$ is a line bundle. 
\end{theorem}

\begin{proof}
If $[\rho]$ is the class of a non central representation, then $\dim ( \mathcal{S}_{\zeta, [\rho]}(M)) =1$ by Theorem \ref{theorem_FKLKK}. If $[\rho]$ is the class of a central representation, then $\dim ( \mathcal{S}_{\zeta, [\rho]}(M)) \geq 1$ by Lemma \ref{lemma_preliminary} and $\dim ( \mathcal{S}_{\zeta, [\rho]}(M)) \leq \dim ( \mathcal{S}_{\zeta, \rho}(M^*))= 1$ by Lemma \ref{lemma_Reynolds2} and Theorem \ref{theorem_Ibundle}. So $\dim ( \mathcal{S}_{\zeta, [\rho]}(M)) = 1$ for all $[\rho]\in X_{\SL_2}(M)$ and Lemma \ref{lemma_hartshorne} implies that $\mathscr{L}^M \to X_{\SL_2}(M)$ is a line bundle. 

\end{proof}

\bibliographystyle{amsalpha}
\bibliography{biblio}

\def\cprime{$'$}
\providecommand{\bysame}{\leavevmode\hbox to3em{\hrulefill}\thinspace}
\providecommand{\MR}{\relax\ifhmode\unskip\space\fi MR }
\providecommand{\MRhref}[2]{%
  \href{http://www.ams.org/mathscinet-getitem?mr=#1}{#2}
}
\providecommand{\href}[2]{#2}
\begin{thebibliography}{FTFKB24}

\bibitem[Bar99]{Barett}
J.W. Barrett, \emph{Skein spaces and spin structures}, Math. Proc. Cambridge
  \textbf{126} (1999), no.~2, 267--275.

\bibitem[BFR23]{BaseilhacFaitgRoche_LGFT3}
S.~Baseilhac, M.~Faitg, and P.~Roche, \emph{Unrestricted quantum moduli
  algebras. {I}{I}{I}.}, ArXiv:2302.00396, 2023.

\bibitem[BG02]{BrownGoodearl}
K.A. Brown and K.R. Goodearl, \emph{Lectures on algebraic quantum groups},
  Basel, Birkhauser, 2002.

\bibitem[BL22]{BloomquistLe}
W.~Bloomquist and T.T.Q. {Le}, \emph{{The {C}hebyshev-{F}robenius homomorphism
  for stated skein modules of $3$-manifolds}}, Math. Zeitschrift (2022),
  no.~301, 1063--1105.

\bibitem[Bul97]{Bullock}
D.~Bullock, \emph{Rings of $sl_2(\mathbb{C})$-characters and the {K}auffman
  bracket skein module}, Comentarii Math. Helv. \textbf{72} (1997), no.~4,
  521--542.

\bibitem[BW16]{BonahonWong1}
F.~{Bonahon} and H.~{Wong}, \emph{{Representations of the {K}auffman bracket
  skein algebra {I}: invariants and miraculous cancellations}}, Inventiones
  Mathematicae \textbf{204} (2016), 195--243.

\bibitem[CL22a]{CostantinoLe19}
F.~Costantino and T.T.Q. L\^e, \emph{{Stated skein algebras of surfaces}}, J.
  Eur. Math. Soc. \textbf{24} (2022), 4063--4142.

\bibitem[CL22b]{CostantinoLe_SSkeinTQFT}
\bysame, \emph{{Stated skein modules of $3$ manifolds and {T}{Q}{F}{T}}},
  arXiv:2206.10906, 2022.

\bibitem[DKS23]{DetcherryKalfagianniSikora_SkeinDim}
R.~{Detcherry}, E.~{Kalfagianni}, and A.S. Sikora, \emph{{Kauffman bracket
  skein modules of small $3$-manifolds}}, arXiv:2305.16188, 2023.

\bibitem[FKL19]{FrohmanKaniaLe_UnicityRep}
C.~{Frohman}, J.~{Kania-Bartoszynska}, and T.T.Q. {L{\^e}}, \emph{{Unicity for
  representations of the {K}auffman bracket skein algebra}}, Inventiones
  mathematicae \textbf{215} (2019), no.~2, 609--650.

\bibitem[FKL21]{FrohmanKaniaLe_DimSkein}
C.~{Frohman}, J.~{Kania-Bartoszynska}, and T.T.Q {L\^e}, \emph{{Dimension and
  Trace of the Kauffman Bracket Skein Algebra}}, Trans. A.M.S. \textbf{8}
  (2021), 510--547.

\bibitem[FKL23]{FKL_GeometricSkein}
\bysame, \emph{{Sliced skein algebras and Geometric {K}auffman bracket}},
  arXiv:2310.06189, 2023.

\bibitem[FTFKB24]{FFK_SkeinIrrep}
M.~Farajzadeh-Tehrani, C.~Frohman, and J.~Kania-Bartoszynska, \emph{The
  {K}auffman bracket skein module at an irreducible representation},
  arXiv:2402.17037, 2024.

\bibitem[GJS23]{GunninghamJordanSafranov_FinitenessConjecture}
S.~Gunningham, D.~Jordan, and P.~Safronov, \emph{The finiteness conjecture for
  skein modules}, Invent. Math. \textbf{232} (2023), 301--363.

\bibitem[GJS24]{GanevJordanSafranov_FrobeniusMorphism}
I.~Ganev, D.~Jordan, and P.~Safronov, \emph{The quantum {F}robenius for
  character varieties and multiplicative quiver varieties}, J. of European
  Mathematical Society (2024), 1--62.

\bibitem[Har77]{Hart}
R.~Hartshorne, \emph{Algebraic geometry}, Springer-Verlag, New York, 1977,
  Graduate Texts in Mathematics, No. 52. \MR{0463157 (57 \#3116)}

\bibitem[KKa]{KojuKaruo_Azumaya}
H.~Karuo and J.~{Korinman}, \emph{{Azumaya loci of skein algebras}},
  arXiv:2211.13700.

\bibitem[KKb]{KojuKaruo_RepRSSkein}
\bysame, \emph{{Classification of semi-weight representations of reduced stated
  skein algebras}}, arXiv:2303.09433.

\bibitem[KM]{KojuMurakami_QCharVar}
J.~{Korinman} and J.~Murakami, \emph{{Relating quantum character varieties and
  skein modules}}, arXiv:2211.04525.

\bibitem[Kor21]{KojuAzumayaSkein}
J.~Korinman, \emph{Unicity for representations of reduced stated skein
  algebras}, Topology and its Applications \textbf{293} (2021), 107570, Special
  issue in honor of Daciberg Lima Goncalves' 70th Birthday.

\bibitem[{Kor}22]{KojuMCGRepQT}
J.~{Korinman}, \emph{{Mapping class group representations derived from stated
  skein algebras}}, SIGMA \textbf{18} (2022), no.~64, 35, arXiv:2201.07649.

\bibitem[KQ24]{KojuQuesneyClassicalShadows}
J.~{Korinman} and A.~Quesney, \emph{{Classical shadows of stated skein
  representations at roots of unity}}, Algebraic \& Geometric Topology
  \textbf{24} (2024), no.~4, 2091–2148, arXiv:1905.03441.

\bibitem[Le15]{LeKauffmanBracket}
T.T.Q. Le, \emph{{On {K}auffman bracket skein modules at roots of unity}},
  Algebraic Geometric Topology \textbf{15} (2015), no.~2, 1093--1117.

\bibitem[{Le}18]{LeStatedSkein}
T.T.Q. {Le}, \emph{{Triangular decomposition of skein algebras}}, Quantum
  Topology \textbf{9} (2018), 591--632.

\bibitem[{Mar}16]{MarcheCharVarSkein}
J.~{March{\'e}}, \emph{{Character varieties in $\mathrm{SL}_2$ and {K}auffman
  skein algebras}}, Topology, Geometry and Algebra of low dimensional
  manifold-RIMS Kokyuroku \textbf{No.1991} (2016), 27--42.

\bibitem[Prz99]{Przytycki_skein}
J.H Przytycki, \emph{Fundamentals of {K}auffman bracket skein modules}, Kobe J.
  Math. \textbf{16} (1999), no.~16, 45--66.

\bibitem[PS00]{PS00}
J.H Przytycki and S.~Sikora, \emph{On skein algebras and
  $\mathrm{SL}_2(\mathbb{C})$-character varieties}, Topology \textbf{39}
  (2000), no.~1, 115--148.

\end{thebibliography}

\end{document}